\documentclass[12pt,reqno]{amsart}
\pdfoutput=1
\usepackage{geometry}     % See geometry.pdf to learn the layout options. There are lots.
\usepackage{graphicx}
\usepackage{color,comment}
\usepackage{amsthm}
\usepackage{amsmath}
\usepackage{amsfonts}
\usepackage{enumerate}
\usepackage{amssymb}
\usepackage{epstopdf}
\usepackage{hyperref}
\usepackage{xypic,tikz,graphicx}

\usepackage{pgfplots}
\pgfplotsset{compat=1.11}
\usepgfplotslibrary{fillbetween}
\usetikzlibrary{intersections,patterns}

\pgfdeclarelayer{bg}
\pgfdeclarelayer{fg}
\pgfsetlayers{bg,main,fg}

\newtheorem{theorem}{Theorem}[section]
\newtheorem{proposition}[theorem]{Proposition}
\newtheorem{proposition*}{Proposition}
\newtheorem{lemma}[theorem]{Lemma}
\newtheorem{corollary}[theorem]{Corollary}

\DeclareMathOperator{\Si}{\Sigma}
\DeclareMathOperator{\Area}{Area}
\DeclareMathOperator{\diam}{diam}

\DeclareMathOperator{\genus}{genus}

\newcommand{\R}{\mathbb{R}}
\newcommand{\RP}{\mathbb{RP}}
\newcommand{\eps}{\varepsilon}

\theoremstyle{definition}
\newtheorem{definition}[theorem]{Definition}

\newtheorem{remark}[theorem]{Remark}

\title[Waist inequality for 3-manifolds]{Waist inequality for 3-manifolds \\with
 positive scalar curvature}

\author{Yevgeny Liokumovich}
\author{Davi Maximo}

\begin{document}
	
	\maketitle
	
	\begin{abstract}
We construct singular foliations of compact three-manifolds $(M^3,h)$ with scalar curvature $R_h\geq \Lambda_0>0$ by surfaces of controlled area, diameter
and genus. This extends Urysohn and waist inequalities
of Gromov-Lawson and Marques-Neves.
	\end{abstract}
	
\section{Introduction}\label{sec:Introduction}
Let $(M^3,h)$ be a compact three-manifold with 
%Heegaard genus $g$ and 
positive scalar curvature $R_h\geq \Lambda_0>0$. We will show that $M$ admits a singular foliation by surfaces of controlled size.

\begin{theorem}\label{thmmaincompact}
There exists a Morse function $f:M \rightarrow \R$,
such that for every $x \in \R$ and each connected
component $\Sigma$ of $f^{-1}(x)$ we have 
\begin{itemize}
    \item[(a)] $\Area(\Si) \leq  \frac{112 \pi}{\Lambda_0}$ if $M$ is non-orientable and $\Area(\Si) \leq  \frac{96 \pi}{\Lambda_0}$ if $M$ is orientable;
    \item[(b)] $\diam(\Si) \leq   \sqrt{\frac{2}{3}} \frac{26 \pi}{\sqrt{\Lambda_0}}$;
%	\item[(c)] $\genus(\Sigma) \leq 7 g$.
	\item[(c)] $\genus(\Sigma) \leq 2$.
\end{itemize}
Additionally, if $M$ is homeomorphic to $S^3$ or connected sum of $S^1 \times S^2$'s
then we can assume $\genus(\Sigma) =0$.
\end{theorem}

%In particular, when $M \cong S^3$ we have that $f^{-1}(x)$ is a disjoint union of 2-spheres of controlled area and diameter.
%Better constants can be obtained for foliations parametrized by a graph.

%\begin{theorem} \label{thmmaincompact2}
%There exists a continuous map $g: M \rightarrow G$ into 
%a graph $G$, such that for every $x \in G$
%\begin{itemize}
%    \item[(a)] $\Area(g^{-1}(x)) \leq  \frac{84 \pi}{\Lambda_0}$ if  is non-orientable and $\Area(\Si) \leq  \frac{72 \pi}{\Lambda_0}$ if $M$ is orientable;
%    \item[(b)] $\diam(g^{-1}(x)) \leq   \sqrt{\frac{2}{3}} \frac{26 \pi}{\sqrt{\Lambda_0}}$;
%	\item[(c)] $\genus(\Sigma) \leq 7 g$.
%	\item[(c)] $\genus(g^{-1}(x)) \leq 2$.
%\end{itemize}
%\end{theorem}

For the area and genus only we obtain nearly 
sharp bounds in Theorem \ref{sharp_area}.
The diameter $\diam(\Si)$ is the extrinsic diameter of $\Si$:
$$\diam(\Si) = \sup \{ dist_{M}(x,y)| x,y \in \Si \}$$
We expect that a similar result should hold
for foliations with controlled intrinsic diameter.

%Theorem resolves the 3-dimensional compact case of a waist inequality conjectured by
%Gromov (\cite[Conjecture $A_{++}$]{Gr2018}).

Existence of a map into a graph with 
the diameter of fibers bounded 
by $\frac{const}{\sqrt{\Lambda_0}}$
%$\frac{4 \pi \sqrt{6}}{\sqrt{\Lambda_0}}$
%$\frac{2 \pi \sqrt{6}}{\sqrt{\Lambda_0}}$ 
was proved
in \cite[Corollary 10.11]{GrLa83} and \cite{Gr2020} (see also \cite{ChLi20}). For manifolds with $Ric_g>0$ and scalar curvature $R_g>2$
Marques and Neves proved a sharp 
bound on the area of maximal fiber in an optimal Morse foliation
of $M^3$ \cite{MaNe12}. 
In \cite{Gr2018} Gromov suggested that by combining methods of 
\cite{GrLa83} and \cite{MaNe12} it should be possible
to obtain a foliation with control on diameter and area
of connected components of the fibers. We show that this is indeed the case.

Note that it is not possible to obtain a version of Theorem
\ref{thmmaincompact} with area or diameter of the connected component of 
%$\Si_t$
$f^{-1}(t)$
replaced by area or diameter of the whole fiber $f^{-1}(t)$.
%$\Si_t$.
Indeed, one can construct manifolds that are hard to cut by considering Gromov-Lawson connect sums of $3$-spheres corresponding to large trees or expander graphs
(see \cite{Mo15}, \cite{PaSw15}).

Combining Theorem \ref{thmmaincompact}
with arguments from \cite{LiZh18} we obtain
a foliation of $M$ by 1-cycles of controlled length.

\begin{theorem} \label{thm: 1-sweepout}
Let $(M^3,g)$ be a closed three-manifold with
 positive scalar curvature $R_g\geq \Lambda_0>0$.
There exists a map $f: M \rightarrow \R^2$, such that 
$$\text{\rm length}(f^{-1}(x)) \leq  \frac{4000}{\sqrt{\Lambda_0}}$$
for all $x \in \R^2$.
\end{theorem}

By applying the min-max arguments (see \cite{Pit76}, \cite{NaRo04}) to the family of curves
constructed in the proof of Theorem \ref{thm: 1-sweepout} we obtain that $M$ contains a stationary 
geodesic net of length bounded
by $ \frac{4000}{\sqrt{\Lambda_0}}$ .

If we do not require that fibers have
controlled diameter, then we can get the following better bounds
for the genus and the area
using estimate of Marques and Neves
for the area of index $1$ minimal surface \cite{MaNe12} 
(cases (a) and (b)
were observed in \cite[3.13B]{Gr2019}).

\begin{theorem} \label{sharp_area}
Suppose $(M,h)$  has positive scalar curvature $R_h\geq \Lambda_0>0$. For every $\eps>0$ the following holds.

\begin{itemize}
    \item[(a)] If $M$ is homeomorphic to $S^3$
then there exists a Morse function $f: M \rightarrow \R$,
such that every non-singular fiber $f^{-1}(t)$ is a union of 
2-spheres of area at most $\frac{24 \pi}{\Lambda_0} + \eps$. 

\item[(b)]  If $M$ is homeomorphic to $S^2 \times S^1$, 
then there exists a Morse function $f: M  \rightarrow S^1$,
such that every non-singular fiber $f^{-1}(t)$ is a union of 
2-spheres of area at most $\frac{24 \pi}{\Lambda_0}+ \eps$. 

\item[(c)] More generally, there exists a graph $G$ and 
a continuous map $f: M \rightarrow G$, such that every fiber
$f^{-1}(t)$
has genus $\leq 2$ and area $\leq \frac{32 \pi}{\Lambda_0}+\eps$.
\end{itemize}
 
\end{theorem}

Our paper is organized as follows. In section 2, we use the min-max existence theory of minimal surfaces to decompose any closed three-manifold with positive scalar curvature into regions whose boundary components are minimal surfaces of controlled area, diameter and with
genus bounded by the Heegaard genus of the prime factors of
$M$, or, for non-oreitnable prime factors, minimal projective planes. In section 3, we recall some facts about mean curvature flow with surgeries and use them to propagate area bounds from the boundary of such region to a foliation parametrized by a graph. In the same section we prove Theorems \ref{sharp_area} and
\ref{thm: 1-sweepout}. Finally, in Section 4, we show how to modify such foliation by a cut-and-paste procedure so that we may guarantee, simultaneously,  controlled area, genus, and diameter; and show how the main result follows from this. 

%Finally, we note that Chodosh and Li in \cite{ChLi20} have recently constructed a method can be used to show that a compact 3-manifold with positive scalar curvature can be mapped to  graph by a map with pre-images of uniformly bounded diameter.

\vspace{0.1in}
\textbf{Acknowledgements} We would like to thank the Institute of Advanced Study and Fernando C. Marques for organizing the 2018-2019 Special Year on Variational Methods in Geometry, during which our collaboration started. We are grateful to Misha Gromov for suggesting Theorem \ref{sharp_area} to us and numerous other valuable suggestions. We are grateful to  Robert Haslhofer for useful conversations. We want to thank Zhichao Wang for pointing out  mistakes in an earlier version of this paper and for multiple valuable comments. Y.L. was supported by NSERC Discovery grant and NSERC Accelerator Award. D.M. was supported by NSF grant DMS-1737006, DMS-1910496, and a Sloan Fellowship.

\section{Minimal surfaces and decomposition of 3-manifolds}
\subsection{Area and diameter estimates.}
We summarize known area and diameter bounds for two-sided minimal surfaces on manifolds of uniformly positive scalar curvature $R_g\geq \Lambda_0$.

	 \begin{theorem}[Area Estimates]\label{area} 
      Let $(M^3,g)$ be a manifold with scalar curvature $R_g\geq \Lambda_0>0$ and suppose $\Sigma\subset M$ is a closed two-sided minimal surface. 
%      If $\Sigma$ is non-orientable let $\tilde{\Sigma}$ denote the double cover      of $\Sigma$.
      
      \begin{itemize}
      	\item[(a)] Suppose $\Sigma$ is stable. Then $\Sigma \cong S^2$ and $\Area(\Sigma) \leq \frac{8\pi}{\Lambda_0}$
      	or $\Sigma \cong \RP^2$ and $\Area(\Sigma) \leq \frac{4\pi}{\Lambda_0}$.
%      	If $\Sigma$ is non-orientable, then 
%      	$\Area(\Sigma) \leq \frac{12 \pi}{\Lambda_0}+\frac{4 \pi}{\Lambda_0}$
%      	\item[(b)] If $\Sigma$ is closed and has index 1, then $|\Sigma| \leq \frac{32\pi}{\Lambda_0}$.
    \item[(b)] If $\Sigma$ is orientable and has index 1, then $|\Sigma| \leq \frac{32\pi}{\Lambda_0}$.
    \item[(c)] If $\Sigma$ is a minimal 2-sphere of index 1, then 
    $|\Sigma| \leq \frac{24\pi}{\Lambda_0}$.
    \item[(d)] If $\Sigma \cong \RP^2$ and has index $1$,
    then $|\Sigma| \leq \frac{28\pi}{\Lambda_0}$.
      \end{itemize}
     \end{theorem} 
     
\begin{proof} For (a)-(c), see  \cite[Proposition A.1]{MaNe12}. For (d), consider Gauss' equation:
$$\textrm{Rc}(\nu,\nu)=\frac{R}{2}-K-\frac{|A|^2}{2}.$$
So that:
\begin{align*}
\frac{\Lambda}{2}|\Sigma| &\leq \int_\Sigma \frac{R}{2} = \int_\Sigma \textrm{Rc}(\nu,\nu)+\frac{|A|^2}{2}+K\\ 
&\leq \int_\Sigma K + \int_\Sigma \textrm{Rc}(\nu,\nu)+{|A|}^{2}\\
& = 2\pi \chi(\mathbb{RP}^2)+ \int_\Sigma \textrm{Rc}(\nu,\nu)+{|A|}^{2}.
\end{align*}

By Yau \cite{Yau87}, since $\Sigma$ is two-sided and has index 1: 
\begin{align*}
    2V_c(n,\Sigma)=2\inf_F \sup_{g} |g\circ F (\Sigma)| \ge \int_\Sigma \textrm{Rc}(\nu,\nu)+{|A|}^{2},
\end{align*}
where $F$ is conformal map from $\Sigma$ to $n$-sphere $S^n$ and $g$ is a conformal automorphism of $S^n$.

Finally, by Li-Yau \cite{LiYau82}: $\inf_n V_c(n,\Sigma) = V_c(\mathbb{RP}^2)=6\pi$. Thus: 
\begin{align*}
    |\Sigma| \leq \frac{28\pi}{\Lambda}.
\end{align*}
\end{proof}
     
We will also need a free boundary version of this estimate that can be found in Ambrozio \cite{Amb15}.

\begin{theorem} \label{area free boundary}
Let $(M^3,g)$ be a manifold with boundary and scalar curvature $R_g\geq \Lambda_0>0$. Suppose $D$ is a properly embedded free boundary stable minimal disk whose boundary lies on a mean convex boundary component of $\partial M$. Then, $\Area(D)\leq \frac{4\pi}{\sqrt{\Lambda_0}}$.
\end{theorem}

The following estimate follows from 
\cite{SY83}, \cite[Theorem 10.2]{GrLa83}, \cite{Gr2020}.

%We use the following filling radius estimate 
%from \cite[Theorem 10.2]{GrLa83} (the estimate also follows from \cite{SY79} and \cite{SY83}, as in \cite[Proposition 2.2]{LiZh18}).

%\begin{theorem}[Filling Radius Estimate]\label{filling}
%Let $(M^3,g)$ be a manifold with scalar curvature $R_g\geq \Lambda_0>0$ and $\Sigma\subset M$ a closed embedded two-sided stable minimal surface. Suppose $\Omega\subset \Sigma$ is a domain in $\Sigma$, and let $\rho>0$ be a number such that
%\begin{itemize}
%	\item[(a)] $\Omega(\rho) = \{x\in\Sigma:dist_\Sigma(x,\Omega) \leq \rho\}$ does not meet $\partial \Sigma$.
%	\item[(b)] $\rm{Image}[H_1(\Omega) \rightarrow H_1(\Omega(\rho))] \neq 0$. 
%\end{itemize} 
%Then, 
%$$\rho\leq \dfrac{\pi}{\sqrt{\Lambda_0}}.$$
%\end{theorem}

\begin{theorem}[Inradius Estimate]\label{diameter1}
Let $(M^3,g)$ be a manifold with scalar curvature $R_g\geq \Lambda_0>0$ and suppose $\Sigma$ is a two-sided embedded stable minimal surface of $M$ with boundary $\partial\Sigma $. Then for all $x \in \Sigma$ the intrinsic distance of $\Sigma$, $d_\Sigma$, satisfies:
\[d_\Sigma(x,\partial \Sigma) \leq \sqrt{\frac{2}{3}} \dfrac{2\pi}{\sqrt{\Lambda_0}} \]
\end{theorem}

\begin{corollary}[Diameter Estimate]\label{diameter2}
	Let $(M^3,g)$ be a manifold with scalar curvature $R_g\geq \Lambda_0>0$. Suppose $\Sigma$ is a closed embedded two-sided minimal surface on $M$.
	\begin{itemize}
		\item[(a)] If $\Sigma$ is stable, then $\rm{diam}_\Sigma \leq \sqrt{\frac{2}{3}} \frac{2\pi}{\sqrt{\Lambda_0}} $.
		\item[(b)] If $\Sigma$ has index 1, then $\rm{diam}_\Sigma \leq \sqrt{\frac{2}{3}} \frac{4\pi}{\sqrt{\Lambda_0}} $.
	\end{itemize}
\end{corollary}
\begin{proof}
Part $(a)$ follows directly from the diameter estimate in \eqref{diameter1}. 

Part $(b)$ also follows from \eqref{diameter1}, as in \cite[Proposition 2.2]{LiZh18}. 
Indeed, we can pick two points $p$ and $q$ at a distance $\diam(\Sigma)$ and consider a geodesic ball $B_r(p)$ with $r = \frac{\diam(\Sigma)}{2}$. Since $\Sigma$
has index $1$ either the connected component of $\Sigma \setminus \partial B_r(p)$
that contains $p$ or the connected component that contains $q$ must be stable. 
Then the result follows from the estimate for the filling radius.
\end{proof}

\subsection{Decomposition of 3-manifolds}
Throughout this section, suppose $M^3$ is a closed three-manifold and $h$ metric on $M$ of positive scalar curvature $R_h\geq \Lambda_0>0$. To prove Theorem \ref{thmmaincompact}, note that we may restrict ourselves to the case when the metric $h$ is bumpy \cite{White17}. 
	
We start by cutting $M^3$ along disjoint closed embedded minimal surfaces to obtain the following decomposition result.

\begin{definition} \label{def: prime}
A 3-manifold $N$ with non-empty boundary is
geometrically prime if 
\begin{enumerate}
    \item $N$ is diffeomorphic to a handlebody with some 3-balls removed;
    \item there are no closed embedded minimal surfaces
    in the interior of $N$;
    \item there exists a closed connected component $\Sigma$ 
(that we'll call ``large'') of $\partial N$ that is either mean convex or minimal of Morse index $1$;
    \item $\partial N\setminus \Sigma $ is either empty or a disjoint union 
    of stable minimal 2-spheres.
\end{enumerate}
\end{definition}

We will also need to define a non-orientable version
of a geometrically prime region.

\begin{definition} \label{def: prime non-orientable}
A 3-manifold $N$ with non-empty boundary is
non-orientable geometrically prime if 
\begin{enumerate}
    \item $N$ is diffeomorphic to $\RP^2 \times [0,1]$ with some 3-balls removed;
    \item there are no closed embedded minimal surfaces
    in the interior of $N$;
    \item there exists a closed connected component $\Sigma \cong \RP^2$ 
(that we'll call ``large'') of $\partial N$ that is either mean convex or minimal of Morse index $1$;
    \item $\partial N\setminus \Sigma $ is a disjoint union 
    of one stable minimal projective plane and stable minimal 2-spheres.
\end{enumerate}
\end{definition}

\begin{theorem}[Decomposition Theorem] \label{decomposition}
Let $M^3$ be a closed three-manifold with a bumpy metric $h$
of scalar curvature $Scal(h)\geq \Lambda >0$. 
%Let $g$ denote the maximal Heegaard genus of prime factors 
%in the prime decomposition of $M$ that are not homeomorphic to $S^1 \times S^2$
%and set $g=0$ if all prime factors are homeomorphic to $S^1 \times S^2$
%or $M$ is simply connected.
Then, there exist closed embedded disjoint
%two-sided 
minimal surfaces $S_1,S_2,\ldots, S_n$ on $M$ such that:
\begin{itemize}
    \item[(a)] $genus(S_i) \leq 2$;
    \item[(b)] $M \setminus \bigcup S_i$ is a disjoint
    union of geometrically prime and non-orientable geometrically prime regions.
%	\item[(b)]  $S_1,S_2,\ldots, S_n$ are pairwise disjoint;
%%	does not have any
%	stable closed embedded minimal surfaces in the interior;
%	\item[(d)]  Each connected component $K$ of $M \setminus \bigsqcup S_i$
%	has exactly one boundary component of Morse index $1$, and all other
%	boundary components are strictly stable minimal surfaces.
	%\item[(c)] Each connected component $N$ of $M\setminus \{\Sigma_1,\Sigma_2,\ldots,\Sigma_k\}$ is a manifold with boundary satisfying $H_2(N,\partial N) =0$ and having exactly one boundary component which is an unstable minimal surface. 
\end{itemize}
\end{theorem}

%In fact, it follows from the maximum principle for mean curvature flow that condition (\ref{c}) can be strengthened to $K$ not having any closed minimal surfaces in its interior.

%\begin{remark} \label{rk:heegaard}
%Note that it follows from Thurston's Elliptization Conjecture
%proved by Perelman that every irreducible 3-manifold of positive
%scalar curvature is a Seifert fiber space with at most $3$
%exceptional fibers and therefore admits a Heegaard splitting 
%of genus less than or equal to $2$ (see \cite{Mor88}, \cite{BCZ91}).
%Therefore, we can replace $g$ by $2$ in the statement of the theorem.
%\end{remark}

\begin{proof}
Since the scalar curvature is positive
the only orientable two-sided stable minimal surfaces in 
$M$ are 2-spheres.
Let $S_1^1, \ldots, S_{k_1}^1$ be a maximal collection of pairwise
disjoint two-sided stable minimal 2-spheres and projective planes. Since the metric is bumpy 
the set of such spheres and projective planes is finite (possibly empty).

Each connected component $K_l$ of $M \setminus \bigcup_{i=1}^{k_1} S_i^1$
is irreducible, since otherwise
we could find a stable minimal sphere in its interior by \cite{MeSiYa82}.

First consider orientable components $K_l$.
By classification of 3-manifolds
of positive scalar curvature it follows that each 
orientable component $K_l$ 
is a spherical space form with possibly some balls removed (all the boundary components are stable minimal 2-spheres). 
%(We don't have any  $S^1 \times S^2$ factors since they will contain a minimal 2-sphere).
%Each $\RP^3$ factor will contain a stable minimal $\RP^2$. 

%For each $N_j$ that is homeomorphic to $S^1 \times \RP^2$ with some balls removed we consider a torus $T$ that intersects each
%slice $\{ t \} \times \RP^2$ in a simple closed curve
%that is not contractible in $\{ t \} \times \RP^2$.
%Minimizing in the isotopy class of $T$ we obtain 
%a stable one-sided torus $T'$ in $N_j$. Cutting $N_j$ along $T'$ we obtain $N_j'$, which is homeomorphic to a solid torus with possibly some balls (which are minimal stable boundaries) removed. The non-spherical boundary component of $N_j'$ is an index 1 minimal torus
%(the double cover of $T'$). 
%Next, for each $N_j$ that is homeomorphic to $\RP^3$ with some balls removed,
%we cut it along a stable minimal projective plane.

%Let $S_{1}^2, \ldots, S_{k_2}^2$ denote the collection of 
%minimal projective planes and tori from the paragraph above. 

%Then each connected component $K_l$ of $\bigcup N_j \setminus \bigcup_{i=1}^{k_2} S_i^2$
%is homeomorphic to some spherical space form other than $\RP^3$
%or a solid torus
%with possibly some balls removed. Note that if $K_l$
%has an index $1$ boundary sphere (that was obtained
%by cutting along a stable projective plane at the previous step),
%then it must be diffeomorphic to a subset
%of a 3-ball with all other boundary components
%stable 2-spheres. 

% note: components that are not RP^3 do not
% contain a projective plane. why?
% because we can take the double cover and 
% minimize to obtain a minimal 2-sphere to cut

If $K_l \cong \RP^3$ with some balls removed, then by \cite[Theorem 17]{KeLiSo} either there is 
an index $1$ minimal Heegaard torus, or there is a stable minimal $\RP^2$
with stable double cover (note that
although \cite[Theorem 17]{KeLiSo} is stated for closed manifolds
the proof applies to the case of a manifold with boundary consisting
of stable minimal spheres). In the second case we cut $K_l$
along the minimal stable $S^2_l \cong \RP^2$ to obtain a 3-sphere with some balls removed
and such that all boundary components are stable.
Applying min-max argument we obtain an index $1$ 2-sphere $S^3_l$ in the interior of $K_l \setminus S^2_l$.

It follows from Thurston's Elliptization Conjecture
proved by Perelman that every orientable irreducible 3-manifold of positive
scalar curvature is a Seifert fiber space with at most $3$
exceptional fibers and therefore admits a Heegaard splitting 
of genus less than or equal to $2$ (see \cite{Mor88}, \cite{BCZ91}).
%Therefore, we can replace $g$ by $2$ in the statement of the theorem.

Suppose $K_l \not\cong \RP^3$ is homeomorphic to a spherical space form
with some balls removed.
We have that the Heegaard genus $g_{K_l}$ of $K_l$
satisfies $g_{K_l} \leq 2$.
% this follows by Haken's theorem
% that reducible 3-manifolds have reducible Heegaard splittings
By \cite[Theorem 17]{KeLiSo} there exists a minimal surface $S^3_l \subset K_l$
of Morse index 1 that is
isotopic to the Heegaard splitting of $K_l$.

Now consider non-orientable components.
If $K_l$ is not orientable then it must be homeomorphic to $\RP^2 \times S^1$ with some balls
removed by \cite[Theorem 5.1]{Ep1961}. In this case we cut along a maximal collection of disjoint stable and index 1
projective planes and 2-spheres to obtain
a collection $K_{l,j}$ each homeomorphic
to $\RP^2 \times [0,1]$ with some balls removed or $S^3$ with some balls removed. If $K_{l,j}$ is homeomorphic to $S^3$ with some balls removed, then it is geometrically prime. We claim that if $K_{l,j}$ is homeomorphic
to $\RP^2 \times [0,1]$ with some balls removed then it is non-orientable geometrically prime. Indeed, if it has an index $1$ two-sphere as one of the boundary components, then we can minimize in the isotopy class (\cite{MeSiYa82}) to obtain a stable minimal 2-sphere in the interior.
%, which is a contradiction as we already cut along a maximal collection of stable 2-sphere. 
Similarly, if it has two or more index 1 projective planes as boundary components, then we can obtain a stable minimal projective plane in the interior.
If all boundary components are stable then by the min-max construction of
\cite{KeLiSo} there exists an index $1$ minimal $\RP^2$ in the interior of $K_{l,j}$.
This contradicts maximality of the collection of minimal spheres and projective planes.
%since $\RP^2 \times S^1$ has a strongly irreducible 
%Heegaard splitting of genus 2 (see \cite{Ochiai}, \cite{RuSc}), by \cite[Theorem 13]{KeLiSo} there exists an
%index 1 minimal $S^3_l$ surface isotopic to the Heegaard surface in $K_l$.

Hence, each connected component $K$ has exactly one boundary component 
$S \subset \partial K$ that is a minimal surface of index $1$.
We claim that $K$ now has no minimal surfaces
in its interior. 
%Let $K$ denote one of the connected components of $K_l \setminus S^3_l$. 
Consider a mean convex surface $\Sigma$ obtained by small perturbation of $S$ to the inside of $K$. Then the level set flow applied to $\Sigma$ will either 
become extinct in finite time or 
will converge to a disjoint collection
of finitely many stable minimal surfaces as
$t \rightarrow \infty$ by the result of White \cite{White00}. 
Since $K$ is homeomorphic to a handle body or $\RP^2 \times [0,1]$ with
some balls removed and 
we already cut the manifold along a maximal collection
of disjoint stable minimal spheres and projective planes we have that the flow converges
to $\partial K \setminus S$.
By the Maximum Principle \cite{Ilm92} $K$ has no smooth embedded minimal surfaces in its interior.
%Hence, the union of surfaces $S^j_i$, $j=1,2,3$, gives the desired collection
%of minimal surfaces of Morse index $0$ or $1$.
\end{proof}

In our proof, geometrically prime regions will be decomposed further along stable minimal disks with free boundary along the large components. 
The following Proposition follows from classical results of Meeks and Yau \cite{MeYa80, MeYa82}:

\begin{proposition} \label{free boundary disc}
Suppose $N$ is a geometrically prime region with large component $\Sigma$ which is mean convex and has genus at least 1. Let $\gamma \subset \Sigma$ be a non-contractible simple closed curve. Then, there exists a properly embedded free
boundary stable minimal disk $D$ with $\partial D \subset \Sigma$ homotopic to $\gamma$.
\end{proposition}

%\begin{proof}
%The case where $N$ is a handlebody with mean convex boundary follows by now classical results of Meeks and Yau \cite{MeYa80, MeYa82}.

%When $N$ is diffeomorphic to a handlebody with some balls removed, those balls all have stable minimal spheres as boundary, and we thus may extend $N$ across those spheres to obtain a mean convex domain $N'$ such that $N\subset N'$ and $N'\setminus N$ is foliated by mean convex spheres. Note that the large component $\Sigma$ is still part of the boundary of $N'$. We then apply Meeks-Yau minimization scheme for disks $D$ with boundary $\partial D$ in $\partial \Sigma$ representing a non-trivial loop in $\pi_1(\Sigma)$. Since $N'\setminus N$ is foliated by mean convex spheres, the minimizer will not intersect that region and thus be contained in $N$. By the maximum principle, it also not touch the spheres in $\partial N \setminus \partial \Sigma$, and thus is a properly embedded free
%boundary stable minimal disk $D$ with $\partial D$ in $\Sigma$.

%\end{proof}

\section{Mean curvature flow and tree foliations}
The foliation in Theorem \ref{thmmaincompact}
is constructed using Mean Curvature Flow with surgery with certain modifications to 
control the diameter. In this section we describe how to construct the foliation
of a geometrically prime or non-orientable geometrically prime region with controlled genus and area. 
In the next section we modify this foliation, so that the diameter
is also controlled.

\subsection{Mean Curvature Flow with surgery.} 
We start with an overview of results from the literature.
We will adopt the terminology of Haslhofer-Kleiner \cite{HaKl17} developed for domains of Euclidean space, which was later adapted to mean-convex domains of general three-manifolds by Haslhofer-Ketover \cite{HaKe19}. Another relevant reference is Brendle-Huisken \cite{BrHui18}.

We recall several definitions and the main results form \cite{HaKl17,HaKe19} that we will use. The first important definition is of a {\it smooth $\alpha$-Andrews} flow, for any $\alpha>0$. These consist, basically, of smooth families of mean convex domains $K_t$, $t \in I\subset \mathbb{R}$, with boundaries $\partial K_t$  moving by mean curvature flow and with the property that $\inf_{\partial K_t} H \geq \frac{4\alpha}{inj(M)}$ and which every point in $\partial K_t$ is $\alpha$-non-collapsed in the sense of Andrews \cite{And12} (See Definition 7.1, \cite{HaKe19}).

The whole point of performing surgeries comes from the fact that such smooth $\alpha$-Andrews cannot be extended indefinitely in time, as singularities can occur. To avoid those, one looks for regions of high-curvature and seeks to replace them before a singularity can form. To make that statement precise, one defines what it means for $K_t$ to have a {\it strong $\delta$-neck} with center $p$ and radius $s$ at time $t_0\in I$. Loosely speaking, this means that an appropriate parabolic rescaling of $K_t$ centered at $(p,t_0)$ is $\delta$-close to the evolution of a solid round cylinder in $D^2\times\mathbb{R}$ with radius 1 at $t=0$ (See Definition 7.2, \cite{HaKe19}). The notion of replacing the final time slice of  a strong $\delta$-neck by a pair of standard caps (surgery) can then be defined rigorously, as in the discussion immediately after Definition 7.2 of \cite{HaKe19}.

With the above terminology, we define a $(\alpha,\delta)$-flow as follows:
\begin{definition}[Definition 1.3, \cite{HaKl17}]\label{def:mcfsmooth}
An $(\alpha,\delta)$-flow $\mathcal{K}$ is a collection of finitely many smooth $\alpha$-Andrews flows $\{K^i_t\}_{t\in[t_{i-1},t_i]}$ $(i=1,2,\ldots,k;t_0<t_1<\ldots<t_n)$ in $M$ such that
\begin{itemize}
\item[(1)] for each $i=1,\ldots,k-1$, the final time slices of some collection of disjoint strong $\delta$-necks
 are replaced by pairs of standard caps as described in Definition 2.4 of \cite{HaKl17}, giving some domains $K^\sharp_{t_i}\subset K^i_{t_i}:=K^{-}_{t_i}$.
 \item[(2)] the initial time slice of the next flow, $K^{i+1}_t:=K^+_{t_i}$ is obtained from $K^\sharp_{t_i}$ by discarding some connected components. 
 %see Figure\footnote{Ask Bob if we can use this picture or put something else} \ref{fig:surgery}.
 \item[(3)] all necks in item (1) have radius $s$ bounded from above and below by a constants depending on $K$.
\end{itemize}
\end{definition}

%\begin{figure}[h]
%\includegraphics{surgery.png}
%\caption{}\label{fig:surgery}
%\end{figure}

When the initial data is mean-convex, one can give a more accurate description of the regions where surgery occurs and also the regions which will be discarded after surgery. This is encompassed in the following notion of an $(\mathbb {\alpha}, \delta, \mathbb{H})$-flow: We say that an $(\alpha,\delta)$-flow $K$ is an $(\mathbb {\alpha}, \delta, \mathbb{H})$-flow if (see Definition 1.17, \cite{HaKe19})

\begin{enumerate}
	\item $\inf H \geq \frac{4\alpha}{inj(M)}$.
	\item Besides the neck parameter $\delta>0$, we have three curvature-scales $H_\textrm{trig} >>
	H_\textrm{ neck} >> H_\textrm{ thick}>>1$, to which we refer as the trigger, neck, and thick curvature. 
	
	\item $H\leq H_\textrm{trig}$ everywhere, and surgery and/or discarding occurs precisely at times $t$ when $H=H_\textrm{trig}$ at some point. 
	
	\item The collection of necks replaced by caps is a minimal collection of solid $\delta$-necks of curvature $H_\textrm{neck}$ which separates the set $\{H\leq H_\textrm{thick}\}$ from $\{H\leq H_\textrm{trig}\}$ in $K^{-}_t$.
	\item The initial condition after a surgery, $K_t^+$ is obtained by discarding precisely those connected components with $H>H_\textrm{thick}$ everywhere. 
%	In particular, of each pair of facing surgery caps, precisely one if discarded. 
\end{enumerate}

With the above terminology, we may now state the main existence result for $(\mathbb {\alpha}, \delta, \mathbb{H})$-flows in \cite{HaKe19}, which, in turn, is an adaptation of the existence theorem in \cite{HaKl17}:

\begin{theorem}[Existence Theorem, {\cite[Thm.~7.7]{HaKe19}}]\label{thm:existencemcf}
Let $K\subset N^3$ be a mean convex domain. Then, for every $T<\infty$, choosing $\delta$ sufficiently small and $H_\text{\rm trig}>>H_\text{\rm neck}>>H_\text{\rm thick}>>1$, there exists an $(\mathbb {\alpha}, \delta, \mathbb{H})$-flow $\{K_t\}$, $t\in[0,T]$ with initial condition $K_0=K$, for some choice of $\alpha=\alpha(K,N,T)$.
\end{theorem}

Next, we have the canonical neighborhood theorem, which gives a description of the regions with high curvature. 

\begin{theorem}[Canonical neighborhood theorem,{\cite[Thm.~7.6]{HaKe19}}]\label{thm:canonicalmcf} For every $\varepsilon>0$ there exists $H_\text{\rm can}<\infty$ such that if $\mathcal K$ is an $(\mathbb {\alpha}, \delta, \mathbb{H})$-flow with $\delta>0$ small enough and $H_\text{\rm trig}>>H_\text{\rm neck}>>H_\text{\rm thick}>>1$, then every spacetime point $(p,t)$ with $H(p,t)\geq H_\text{\rm can}$ is $\varepsilon$-close to either (a) an ancient $\alpha$-Andrews flow in $\mathbb{R}^3$ or (b) the evolution of a standard cap preceded by the evolution of a round cylinder $\overline{D}^2\times\mathbb{R}\subset\mathbb{R}^3$. 
\end{theorem}

A corollary of the canonical neighborhood theorem is the following description of the discarded components by surgeries:

\begin{corollary}\label{cor:discarded}
For $\varepsilon>0$ small enough, any $(\mathbb {\alpha}, \delta, \mathbb{H})$-flow satisfying the hypothesis of Theorem \ref{thm:canonicalmcf} has discarded components diffeomorphic to the three-disk $D^3$ or solid torus $D^2\times S^1$. 
\end{corollary}

\begin{remark}\label{remark:bilip}
By picking the curvature parameters large enough, we may always ensure that, at the neck scale $H^{-1}_{\textrm neck}$, the ambient space looks as close as we want to Euclidean. In this situation, we may also ensure that there will exist a 2-Lipschitz map from any strong $\delta$-neck to the Euclidean cylinder $D^2\times \mathbb R$ of radius $H^{-1}_{\textrm neck}$. In addition, as observed in Corollary 8.9 of \cite{BHH17} and in \cite{BHH19}, each discarded component is either a convex sphere of controlled geometry, a capped-off chain of $\varepsilon$-necks, or an $\varepsilon$-loop. In each case, one may chop them into a collection of spheres, each contained in a small ball that can be mapped to a ball in $\mathbb{R}^3$ of radius $H^{-1}_{\textrm neck}$ by a 2-bilipschitz diffeomorphism. 
\end{remark}

\subsection{Mean Curvature flow in geometrically prime regions}\label{section:geoprimemcf} Let $N$ be a geometrically prime or non-oriented geometrically prime region obtained from Theorem \ref{decomposition}, and let  $\partial N= \Sigma \cup \Sigma_1 \cup \Sigma_2\cup\cdots \cup \Sigma_l$, where $\Sigma$ is the large component and $\Sigma_1, \Sigma_2,\ldots, \Sigma_l$ are stable minimal spheres or projective planes (note that $l$ might be zero). Since $\Sigma$ is a two-sided minimal surface of Morse index 1, we may apply a small inward deformation to the inside of $N$ to obtain a mean convex surface  $\Sigma_0$ using the lowest eigenfunction of the
stability operator (see proof of Theorem 3.1 in \cite{HaKe19}). By Theorem \ref{thm:existencemcf}, for $T>0$,  choosing $\delta<<1$ and $H_\text{\rm trig}>>H_\text{\rm neck}>>H_\text{\rm thick}>>1$ there will exist a $(\mathbb {\alpha}, \delta, \mathbb{H})$-flow $K_t$ with $\partial K_0=\Sigma_{0} \cup \Sigma_1 \cup \Sigma_2\cup\ldots \cup \Sigma_l$ for which the canonical neighborhood theorem applies. We prove:

\begin{proposition}\label{claim:dis}
Given $\varepsilon>0$, there exists a large $T>0$ such that, after making $\delta$ smaller and $H_\text{\rm thick}$ larger if necessary, there exists $(\mathbb {\alpha}, \delta, \mathbb{H})$-flow $K_t$ with $\partial K_0=\Sigma_{0} \cup \Sigma_1 \cup \Sigma_2\cup\ldots \cup \Sigma_l$, such that
%\Sigma_{0} \cup N_2 \cup N_3\cup\ldots \cup N_l$ on the interval $[0,T]$ such that:
\begin{itemize}
\item[(a)] The final slice $K_T$ is either empty, if $l=0$, or consists of smooth connected mean-convex surfaces which lie at distance at most $\varepsilon$ from the minimal stable boundaries 
$\partial N \setminus \Sigma$.
%$N_2,\ldots, N_l$. 
\item[(b)] The points of $N$ which are not in $\bigcup_t\partial K_t$, $t\in [0,T]$, are precisely the points belonging to discarded or caped components of the surgery process. In addition, the set 
$$\left\{x\in N\setminus \bigcup_{t\in[0,T]}\partial K_t~\bigg\vert~d(x,\partial N) \geq \varepsilon \right\}$$
has volume at most $\varepsilon$.
\end{itemize}
\end{proposition}

\begin{proof}
Our argument is in the vein of Theorem 8.1 in \cite{HaKe19}. Let us first consider the case when $l=0$, that is, $\partial N=\Sigma$. Since there are no closed stable minimal surfaces in the interior of $N$, by a result of White \cite{White00}, the level-set flow must become extinct in finite time, say $T$; that is, $K_T = \emptyset$. Because any mean curvature flow with surgery starting at $K_0$ is also a family of closed sets that is a set-theoretic subsolution for the level-set flow, in the terminology of Ilmanem \cite{Ilm94}, and the level-set flow $M_t$ is the maximal set-theoretic subsolution, we get an a priori bound $T<\infty$ for the extinction time of any mean curvature flow with surgeries starting at $K_0$. Thus, given $\varepsilon>0$, we may select $\delta<<1$ and $H_\text{\rm trig}>>H_\text{\rm neck}>>H_\text{\rm thick}>>1$ so that the $(\mathbb {\alpha}, \delta, \mathbb{H})$-flow $K_t$ starting at  $\partial K_0 = K$ will be sufficiently close to the level-set flow in the Hausdorff sense so that all the discarded regions by the flow will have area summing up to at most $\varepsilon$. 

Now, we assume $l\geq1$. We fill in $N$ by gluing three-disks to the boundaries $\Sigma_2,\ldots,\Sigma_l$ while maintaining their minimality and stability. By doing this we obtain a domain $M$ with mean-convex boundary $\partial M =\Sigma_0$ which contains stable minimal surfaces $\Sigma_1,\Sigma_2,\ldots,\Sigma_l$, but no other stable minimal surfaces in the region between $\Sigma_0$ and $\Sigma_1,\Sigma_2,\ldots,\Sigma_l$.
Consider the level-set flow (\cite{CGG91,ES91}) $\{M_t\}_{t\geq0}$ starting at $M_0=M$ and let 

$$M_\infty : = \bigcap_{t\geq 0} M_t.$$
By the barrier principle, since $\Sigma_1,\Sigma_2,\ldots,\Sigma_l$ are minimal, they will be contained in $M_t$ for all $t\geq 0$ and thus in $M_\infty$. Additionally, by Theorem 11.1 of White \cite{White00}, $M_\infty$ has finitely many connected components and the boundary of each one of them is a stable minimal surface. Since there is no stable minimal surface in region between of $\partial M$, and $\Sigma_1, \Sigma_2, \ldots , \Sigma_l$, we conclude that
$$\partial M_{\infty}=\Sigma_1\cup \Sigma_2\cup \cdots \cup \Sigma_l.$$ 

By the same theorem of White \cite{White00}, the convergence of $\partial M_t$  to $\partial M_{\infty}=\Sigma_1\cup \Sigma_2\cup \cdots \cup \Sigma_l$ is smooth as $t\nearrow\infty$.

Thus, given $\varepsilon>0$, we may pick $T>0$ such that $\partial M_T$ consists of a union of graphical surfaces on a $\varepsilon/2$-neighborhood of $\partial M_{\infty}=\Sigma_1\cup \Sigma_2\cup \cdots \cup \Sigma_l$. We consider a sequence of $(\mathbb {\alpha}, \delta_n, \mathbb{H}_n)$-flows $M^n_t$ on $[0,T]$ with initial condition $\partial M^n_0=\partial M = \Sigma_{0}$, $\delta_n,\nearrow\infty$ and thick curvature thresholds  $H^n_\text{\rm thick}\nearrow\infty$. 

Since the initial domain is kept fixed and the thick curvatures tend to infinity as $n\rightarrow\infty$, by the work of Lauer \cite{La13}, we have that the above sequence of flows with surgery converge to the level set flow in the Hausdorff sense. 
%https://www.overleaf.com/project/5d67c29c9d0b4b52da4bbe34
Thus, picking $n$ sufficiently large, we have an $(\mathbb {\alpha}, \delta_n, \mathbb{H}_n)$-flow that satisfy the desired hypothesis. 
\end{proof}

\subsection{Tree foliations.}
To prove Theorem \ref{thmmaincompact}, it will be convenient
to introduce a certain type of foliations of geometrically prime manifolds that are parametrized by a $1$-dimensional oriented tree $G$. This definition can be suitably extended to general three-manifolds by considering more general graphs, but we will not use such definition and will skip it entirely. 

\begin{definition} \label{def: tree}
Let $G$ be an oriented tree with vertices of degree $1$, $2$ or $3$. A family of surfaces $\{\Sigma_t\}_{t\in G}$ is a \emph{tree foliation} of
$N$ if there exists a 
continuous map $p: N \rightarrow G$ and continuous orientation preserving 
map $ s: G \rightarrow [0,T]$, such that
\begin{itemize}
    \item[(a)] for each $x \in G$ surface $\Sigma_x = p^{-1}(x)$ is connected;
%    \item[(b)] for each $x \in G$ the surface $s^{-1}(x)$ is connected;
    \item[(b)] $s^{-1}(0)$ is the unique vertex $v_r$, called the root, of $G$
    of degree $1$ that has an edge directed away from it;
%    surface $\Sigma_0$ corresponding to the root vertex is a connected minimal surface of Morse index $1$;
    \item[(c)] If $v$ is a vertex of degree $1$, then
    $\Sigma_v$ is a point or a connected component of 
    the boundary of $N$;
    \item[(d)] Let $\mathring{E}$ denote the interior of 
    an edge $E$ of $G$. Then for each edge $E$ the family $\{\Sigma_t\}_{t \in \mathring{E}}$
    gives a smooth foliation of $p^{-1}(\mathring{E})$;
    \item[(e)]\label{e:tree} For each vertex $v$ of degree
    $d=2$ or $3$ 
%    there is a smooth closed curve 
%    $\gamma \subset \Sigma_v$,
%    such that 
%    $\Sigma_v \setminus \gamma$
    $\Sigma_v$ is a
    union of a smooth closed connected surface $S_1$ and a 
    smooth
    connected surface $S_2$ with boundary $\gamma \subset S_1$.
%    meeting 
%    along finitely many disjoint closed curves $\gamma= \bigsqcup \gamma_j$. 
    If $E$ is an edge adjacent to
    $v$, then away from $\gamma$ surfaces $\Sigma_t$, $t \in E$, converge
    smoothly and graphically to a subset of 
    $\Sigma_v$ as $t \rightarrow v$.
%    \item[(d)] The graph distance 
\end{itemize}
\end{definition}

%and have mean convex boundary for every value of $t$ which is not the image of a vertex of $\Gamma$ by $h:\Gamma\rightarrow \mathbb{R}$. 

We use mean curvature flow with surgery to prove existence of a tree foliation of a geometrically prime 3-manifold with area and genus controlled in terms of area and genus of the large boundary component. Additionally we would like surfaces in the foliation to satisfy certain conditions on their mean curvature. 

\begin{definition}
Let $\{\Si_x \}_{x \in G}$
be a tree foliation of $N$ with $p:N \rightarrow G$, $s:G \rightarrow [0,T]$ as in Definition \ref{def: tree}.  We say that the {\it foliation is mean convex on} $U \subset [0,T]$
if for every $t \in U$ the set 
$(s\circ p)^{-1}(t)$ consists of a disjoint union of smooth mean-convex surfaces with mean curvature vector positively oriented with respect to $s:G \rightarrow [0,T]$.
\end{definition}

\begin{definition}
Given any $\varepsilon>0$, we will say that a set $U \subset [0,T]$ is $\varepsilon$-{\it small} if $U$ is a union of finitely many disjoint 
open intervals $(a_i, b_i)$ and for each $i$
we have $(s\circ p)^{-1}([0,b_i]) \subset 
N_\varepsilon((s\circ p)^{-1}([0,a_i]))$, where $N_\varepsilon(\cdot)$ denotes the tubular neighborhood taken in $N$ with respect to the metric $h$.
\end{definition}

The main result of this section is the existence
of such foliation on geometrically prime 3-manifolds with the following properties:

\begin{proposition} \label{MCF_tree}
Let $N$ be a geometrically prime or non-oriented geometrically prime
3-manifold and let $\Sigma$ denote
the large boundary component of $N$.
For every $\varepsilon>0$ there exists a tree foliation
$\{\Sigma_x\}_{x \in G}$, such that
\begin{enumerate}
    \item $\genus(\Sigma_x)\leq \genus(\Sigma)$
    \item $\Area(\Sigma_x)\leq \Area(\Sigma)$
\end{enumerate}
Moreover, the foliation is mean convex on $[0,T] \setminus B$
for an $\varepsilon$-small set $B$.
\end{proposition}

\begin{remark}
It should be possible to prove existence
of a tree foliation that is mean convex everywhere on
$[0,T]$, except for finitely many points $t_i \in [0,T]$
(corresponding to images of vertices of $G$ in $[0,T]$)
using techniques from \cite{BHH19} and \cite{HaKe19}.
However, the weaker version stated above is sufficient 
for our purposes.
\end{remark}

\begin{proof} We will start by proving slightly weaker version of the proposition: $(\ast)$ Given $\varepsilon>0$, there exists a tree foliation
$\{\Sigma_x\}_{x \in G}$, such that
\begin{enumerate}
    \item $\genus(\Sigma_x)\leq \genus(\Sigma)$ 
    \item $\Area(\Sigma_x)\leq \Area(S)+\varepsilon$
\end{enumerate}
which is mean convex on $[0,T] \setminus B$ for an $\varepsilon$-small set $B$. 

The proof follows essentially by flowing the large component $\Sigma$ of $\partial N$ by mean curvature flow, which is area-decreasing and does not increase genus (even after surgeries). Because of potential surgeries and components getting discarded, we will need to introduce a procedure to foliate the surgery and discarded regions (without increasing the genus or area by more than $\varepsilon$).

First, let's set up the mean curvature flow with surgeries. We use the notation and apply the existence results of Section \ref{thm:existencemcf} on mean curvature flow with surgeries starting at the large component of $\partial N$. Let $\{K^i_t\}$, $i=1,\ldots,k+1$, be the smooth flows that comprise the flow with surgeries $K_t$ (Definition \ref{def:mcfsmooth}) and $0<t_1<...<t_k<T-\varepsilon$ the respective surgery times. We may assume $\varepsilon$ is sufficiently small and $T$ is chosen sufficiently large so that $\partial K_{T-\varepsilon} \setminus \partial N$
is a disjoint union of spheres, each lying in a small 
tubular neighborhood and being graphical over a stable minimal
surface in $\partial N\setminus \Sigma$. We may extend the family to the interval $[T-\varepsilon,T]$ by defining a smooth isotopy of each connected component of
$\partial K^k_{T-\varepsilon} \setminus \partial N$
to the corresponding minimal surface $\Sigma_1,\ldots, \Sigma_l$.

We now give an informal description of 
how we derive the foliation $\{ \Sigma_x\}_{x \in G}$.
First consider $t \in [0,T]$ outside of small
neighborhood of surgery times $t_i$.
For each family $\{\partial K_t^i \}$, 
we let $\{\Sigma_x\}$ to be smooth families of connected 
components of $\partial K_t^i \setminus \partial N$
parametrized by a disjoint union of intervals $E_1^i, ..., E_{n_i}^i$. From these intervals we start building edges of the graph $G$, which for the moment are disconnected. 

At a surgery time, the replacement of a neck with two standard caps occurs in a small ball, where the metric is nearly Euclidean
and the neck is close to a standard cylinder.
For each such replacement we add a vertex $v$ to the graph $G$ and define surface $\Sigma_v$ that is equal to the the union
of the connected component that contains the neck and a small disc $D$ in the middle of the neck/cylinder. This process starts connecting the edges of the graph $G$ corresponding to the intervals $E_1^i, ..., E_{n_i}^i$: we define deformations of the unions of each of the half-cylinders with $D$ into the corresponding standard cap. Proceeding this way for each neck-caps replacement we obtain
the union of connected components of $\partial K^{i+1}_t$
for some $t$ slightly larger than $t_i$ and discarded components of $\partial K^i_{t}$. For each connected component of the discarded part
we define a tree foliation that ends in a collection of points. This can be done since the discarded part 
has very small area. 

Note that in this procedure, depending on whether the neck-caps surgery disconnects a  component of $\partial K_t^i$, the vertex $v$ will be of degree bigger than 1, and potentially great than 3. To account for this and to have only vertices with degree at most 3, we will change the above procedure slightly by hand by gluing only one disk $D$ at a time for every surgery time where several neck-caps replacements take place. We describe the construction in more detail below.

Let $\sigma>0$ be a small constant sufficiently smaller than $\varepsilon$.
Define maps $p_i: K_{t_i}^+ \setminus K_{t_{i-1}+ \sigma}^i 
\rightarrow \bigcup_j E^i_j$, so that the fibers $\{p_i^{-1}(t) \}$, 
$t \in E^i_j = [t_{i-1}+ \sigma, t_i]$, are a smooth family of connected surfaces evolving
by mean curvature flow. The map $s: \bigcup_j E^i_j \longrightarrow [t_{i-1}+\sigma, t_i]$ is defined so that $\partial K_t^i \setminus \partial N = (s \circ p)^{-1}(t)$.

It remains to describe the family of surfaces
$(s \circ p)^{-1}(t)$ for $t \in B=\bigcup_{i=1}^k [t_i, t_i+\sigma]$,
that is, as the family goes through a surgery. These surgeries can be of two types: neck-caps replacements or discarding components. 

\vspace{0.1in}
\textbf{Vertices corresponding to neck-caps surgeries.}
We define vertices corresponding to surgeries of the mean curvature flow
and the families of surfaces parametrized by
neighborhoods of the vertices in the tree $G$. Fix a surgery time $t_i$. Let
$\{ B_j \}_{j=1}^{m_i}$ be the set of disjoint balls
so that $K^\sharp_{t_i}$ is obtained
by replacing $\delta$-necks in balls
$B_j$ by pairs of standard caps.
$K^{i+1}_{t_i}$ is then obtained from 
$K^\sharp_{t_i}$ by discarding small connected components. 

To ensure the graph $G$ only has vertices with degree at most 3, we will deal with each $B_j$ separately. We pick a separation time $\tau_i \in (0, \frac{\sigma}{2})$
sufficiently small so that
the evolution of smooth mean curvature flow
of $K^{i}_t$ and $K^{i+1}_t$ exists
for $t \in [t_i, t_i+\tau_i]$ and it
is graphical over $K^{i}_{t_i}$ and $K^{i+1}_{t_i}$, correspondingly. By Remark \ref{remark:bilip}, for each $j=1,..., m_i$, we have maps
$\Phi_j: B_j \rightarrow \R^3$, such that:

\begin{enumerate}
    \item $\Phi_j$ is a $2$-bilipschitz 
    diffeomorphism onto its image in
    $\R^3$;
    \item the images of the $\delta$-necks in $\Phi_j(\partial K^{i}_t \cap B_j)$ are contained in $C_t^j$, where $C_t^j$
    is a family of concentric infinite cylinders in $\R^3$ for $t \in [t_i, t_i + \tau_i]$;
    \item the images of the caps in $\Phi_j(\partial K^{i+1}_t \cap B_j)$ are contained in  $S_t^{j,1} + S^{j,2}_t$, where $S^{j,1}_t$ and $S^{j,2}_t$ are two families
    of concentric spheres in $\R^3$ for $t \in [t_i, t_i + \tau_i]$.
\end{enumerate}

Let $t_{i,j} = t_i + \frac{(j-1) \tau_i}{2m_i}$. At $t_{i,1}=t_i$, we focus on the neck replacement taking place in $B_1$. Let $\Gamma$ denote the limiting surface of $(s \circ p)^{-1}(t)$ as $t \rightarrow t^{-}_{i,1}$ and let $C^1 = \Phi_1(\Gamma \cap B_1)$, which is a cylinder by construction (Remark \ref{remark:bilip}). Let $D$ denote a flat disc in $\R^3$ that is perpendicular to the axis of $C^1$ and is equidistant from the two boundary curves of $C^1$. We define a vertex $v_{i,1}$ in the tree 
$G$ and surface $\Sigma_{v_{i,1}}=p^{-1}(v_{i,1})$ to be 
the union of $\Phi_1^{-1}(D)$ and the connected component of
$\Gamma$ that intersects $B_1$.

The disc $D$ separates the cylinder $C^1$ into two components, 
$C_1^1$ and $C_2^1$.
For $t \in [t_{i,1},t_{i,1}+\tau_i]$
and $l=1,2$
we define a family of discs $D^{l,1}_t$
with the following properties:
\begin{enumerate}
    \item $D^{l,1}_{t_{i,1}} = C^1_l \cup D$;
    \item $\partial D^{l,1}_t = \partial S^l_t$;
    \item $\{D^l_t\}$ is a family of smooth 
    discs foliating the region between  
    $C^1_l \cup D$ and $S^l_{t_{i,1}+\tau_i}$.
\end{enumerate}

Finally, for $t \in [t_{i,1},t_{i,1}+\tau_i]=[t_{i,1}, t_{i,2}]$, let $\Gamma_t$ denote the evolution by
mean curvature flow of the union of one
or two connected components of $\partial\left((s \circ p)^{-1}(t_i)\right)$ 
that were obtained from $\Gamma$
by the neck-caps replacement.
Then family 
$$(\Gamma_t\setminus B_1) \cup \Phi_1^{-1}(D^{1,1}_t \cup D^{2,1}_t)$$
is the desired family of surfaces parametrized by a small neighborhood
of vertex $v_{i,1}$ in the graph $G$. For components that belong
to the thick part (that is, the part that is not discarded in 
the MCF with surgery) we extend the family 
forward in time using MCF until they join up with 
families parametrized by edges $E^i_1,E^i_2,\ldots,E^i_{n_i}$.

We then repeat the above process for $t \in [t_{i,2}, t_{i,3}]$ focusing on the neck-caps replacements in $B_2$ and so on. 

\vspace{0.1in}
\textbf{Contracting discarded components.} As noted in Remark \ref{remark:bilip}, the discarded components are either a convex sphere of controlled geometry,  a capped-off chain of $\varepsilon$-necks, or an $\varepsilon$-loop and 
have areas $\rightarrow 0$ as $\varepsilon \rightarrow 0$. 

By coarea inequality we can cover each discarded
component $S$ by balls $B_l$, such that $B_l$
is $2$-bilipschitz diffeomorphic to a small ball in $\R^3$,
and $\partial B$ intersects $S$ in a finite union of small
loops. We start filling these loops by minimal discs
one by one
in a way similar to the procedure described above.
Eventually, we obtain a collection of
disjoint spheres each contained in a small ball
$B_l$. Then by Lemma \cite[Lemma 4.1]{CL2019} there exists a Morse  foliation contracting the sphere to a point.
The argument in Lemma \cite[Lemma 4.1]{CL2019} can be slightly modified to give a tree foliation instead of a Morse foliation.
\vspace{0.1in}

Finally, observe that the foliation is mean convex
on $[0,T]\setminus B$.
Moreover, we may assume that parameters of the
$(\mathbb {\alpha}, \delta, \mathbb{H})$-flow 
have been chosen so that $K_{t_i}^i \subset N_{\varepsilon/2}(K_{t_i}^{i-1})$
and so for $\sigma$ sufficiently small
we have that the set $B$ is $\varepsilon$-small.

\vspace{0.1in}
\textbf{End of the proof.} We now show how to finish the argument for Proposition \ref{MCF_tree} from its weaker version $(\ast)$. Given a geometrically prime or non-orientable geometrically manifold $N$ with large component $\Sigma$, we flow $\Sigma$ by smooth mean curvature flow for a short time $[0,t_0]$ and, by mean convexity, me may assume:
$$\Area(\Sigma_{t_0}) < \Area(\Sigma)-\alpha,$$
for some small real number $\alpha$. We then discard the region of $N$ bounded by $\Sigma$ and $\Sigma_{t_0}$, obtaining a new geometrically prime manifold $N'$ with large component $\Sigma_{t_0}$. Proposition \ref{MCF_tree} then follows by applying its weaker version $(\ast)$ to $N'$ with a choice of $\varepsilon <\alpha$.
\end{proof}

\begin{proof}[Proof of Theorem \ref{sharp_area}]
By \cite{White:bumpy} we can make an arbitrary small perturbation to make the metric on $M$ bumpy.
For an arbitrarily small $\eps>0$ we will construct a foliation for the perturbed bumpy metric,
so that the bounds on the area and diameter in the original metric are worse by at most $\eps$.

When $M$ is homeomorphic to $S^3$ we apply Theorem \ref{decomposition}
to decompose $M$ into a union of geometrically prime regions
that have an index 1 minimal sphere as their large boundary component.
From Proposition \ref{MCF_tree} we obtain a tree-foliation of each 
geometrically prime region by a family of 2-spheres. The areas
of these spheres are at most
$\frac{24 \pi}{ \Lambda_0}$ by Theorem \ref{area}.
We can assemble these foliations together to obtain a map 
$\tilde{f}: M \rightarrow G$. Note that $G$ must be a tree. 
Fix a map $g: G \rightarrow \R$, such that the restriction of $g$
to any edge is linear and non-singular. The map $\tilde{f} \circ g$
can then be perturbed to a Morse function from $M$ to $\R$
with desired properties. The details of this perturbation are postponed
until the proof of Theorem \ref{thmmaincompact},
where it is done in a more general setting.

If $M$ is homeomorphic to $S^2 \times S^1$ we consider
manifold $M' \cong S^2 \times [0,1]$ obtained by cutting $M$
along a minimal sphere that minimizes area in the isotopy class
of $S^2 \times \{0\}$. Applying Proposition \ref{MCF_tree}
as above we obtain a map $\tilde{f}: M' \rightarrow G$ onto a 
tree $G$ with vertices $v_1$, $v_2$ of $G$ corresponding to
two stable boundary spheres of $M'$.
%, $\tilde{f}^{-1}(v_1) \sqcup \tilde{f}^{-1}(v_2) = \partial M$.
Let $g: G \rightarrow S^1$ be a map with $g(v_1) = g(v_2)$
and such that the restriction of $g$ to each edge is non-singular.
A perturbation of $\tilde{f} \circ g$ then gives the desired
Morse function on $M$. 

Finally, case (c) follows by Proposition \ref{MCF_tree}.
%and Remark \ref{rk:heegaard}.
\end{proof}

\begin{proof}[Proof of Theorem \ref{thm: 1-sweepout}]
Let $f:M \rightarrow G$ be the map from $M$ to a graph $G$
from part (c) of Theorem \ref{sharp_area}.
Recall that graph $G$ has vertices of degree 
at most $3$.
%Consider graph $G = M / \sim $, where we identify points 
%in the same connected component of $f$: $x \sim y$ if there
%is a path between $x$ and $y$ contained in $f^{-1}(t)$ for some $t \in \R$.
%Map $f$ naturally induces map $\overline{f}: M \rightarrow G$.

\begin{figure}
\centering
\includegraphics[scale=0.5]{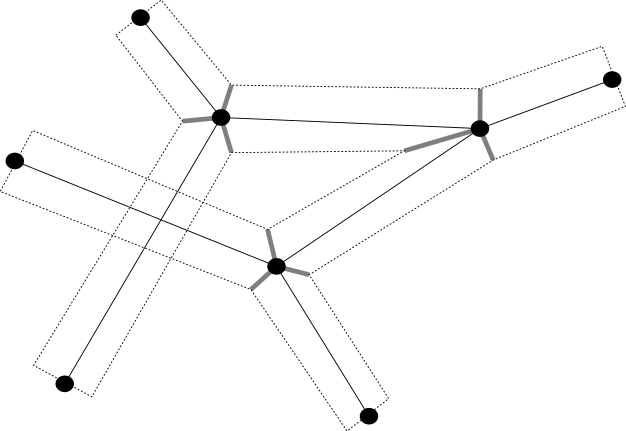}
\caption{Mapping $M$ into the neighborhood
of graph $G$ in $\R^2$.}
\label{fig:1cycles}
\end{figure}

Consider a general position smooth immersion of $G$ in $\R^2$. In particular, vertices are disjoint and 
every edge $E$ of $G$ in $\R^2$ intersects at most finitely many other edges.
Moreover, let $\{ U_{E_i} \}$ denote the set of $\varepsilon$-neighborhoods of
edges of $G$, then for $\varepsilon>0$ sufficiently small we may assume that
each point $x\in \R^2$ is contained in at most two distinct $U_{E_i}$, $U_{E_k}$
if $x$ lies at a distance $>\varepsilon$ from the vertices of $G$.
If $x \in B_{\varepsilon}(v)$ for some vertex $v$ of $G$, then it may 
lie in at most three sets $U_{E_{i_1}}, U_{E_{i_2}}, U_{E_{i_3}}$
corresponding to edges $E_{i_1}, E_{i_2}, E_{i_3}$ adjacent to 
the vertex $v$.

First we define map $F$ in the neighborhood
of a degree $3$ vertex of $G$. Let $T_v=l_1 \cup l_2 \cup l_3$
denote a ``tripod'', a union of three arcs emanating from 
vertex $v$ to the boundary of $\varepsilon$-neighborhood
of $G$, each $l_i$ bisecting the angle between two edges of $E$
(see Fig. \ref{fig:1cycles}).
Recall from the proof of Theorem \ref{MCF_tree}
that $f^{-1}(v)$ is a union of a surface $\Sigma$
and a disc $\Sigma_1$ intersecting $\Sigma$ in a closed curve
$\gamma$ that separates $\Sigma$, $\Sigma\setminus \gamma = \Sigma_2 \cup \Sigma_3$. 
For each $\Sigma_i$ there exists a Morse function $f_i: \Sigma_i \rightarrow [0,1]$
with $f_i^{-1}(0) = \gamma$ and satisfying
$$\textrm{length}(f_i^{-1}(t)) \leq C \sqrt{\Area(\Sigma_i)(\genus(\Sigma_i)+1)}+
\textrm{length}(\gamma)$$
for $t \in [0,1]$ (\cite{GAL}).
(Recall from the proof of 
Theorem \ref{MCF_tree} that the length 
of curve $\gamma$ can be assumed to be 
arbitrarily small).
Hence, we can define map $F$
from $\Sigma_1 \cup \Sigma$ to
the tripod $T$ with the desired
length bounds.

In an analogous way we define $F$
near all other vertices. It remains to
define an interpolation that will
take values in the neighborhood of 
each edge. Existence of such 
an interpolation, with controlled
length of pre-images, follows 
by Proposition 4.3 and Theorem 4.2
from \cite{LiZh18}. 

We now discuss the value of constant $C$.
Surfaces $\{f^{-1}(x)\}$ have
areas bounded by $\frac{32 \pi}{\Lambda}$
and genus $\leq 2$. If we plug that into the estimate 
from Theorem 4.2 in \cite{LiZh18}
for parametric sweepouts, keeping in mind
that graph $G$ can have double points in $\R^2$,
we obtain that the length is bounded by
$\textrm{length}(F^{-1}) \leq \frac{70000}{\sqrt{\Lambda}}$. A better estimate can be obtained if we cut
genus 2 surface along short curves and then apply a better bound for the lengths of curves in a sweepout of a surface that is diffeomorphic to a sphere
with holes from \cite{L14}. 

We briefly describe how to cut 
a genus 2 surface $\Sigma$.
First we find a non-contractible
curve $\gamma$ of length at most
$\sqrt{\frac{2}{\sqrt{3}}}\sqrt{Area(\Sigma)}$
(\cite{KaSa06}). Cutting $\Sigma$
along $\gamma$ we obtain either one
or two genus 1 surfaces and glue in two
spherical caps along $\gamma$
each of area $\textrm{length}(\gamma)^2 / 2 \pi$.
We then subdivide the resulting surface 
(or surfaces) using systolic inequality 
for the torus and slide the subdividing 
curve (or curves) into the interior of the
original surface $\Sigma$ without affecting its
length. In the end we obtain one or two surfaces with boundary diffeomorphic to a sphere with holes. 
%$2 \sqrt{\frac{2}{\sqrt{3}}} \sqrt{\Area(\Sigma)/2+ \frac{1}{2\pi}\frac{2}{\sqrt{3}}}+2 \sqrt{\frac{2}{\sqrt{3}}} \sqrt{\Area(\Sigma)} $
Using the bounds from \cite{L14}
we obtain that there exists a sweepout of $\Sigma$
with curves of length $< \frac{1000}{\sqrt{\Lambda_0}}$.
By Proposition 4.3 \cite{LiZh18} we can construct
a family of sweepouts for the 1-parameter family of 
surfaces corresponding
to an edge of $G$
with lengths $< \frac{2000}{\sqrt{\Lambda_0}}$.
Hence, we can take $C = 4000$.
\end{proof}

\section{Diameter bounds and the proof of Theorem \ref{thmmaincompact}}
\subsection{Diameter bounds for level sets of the distance function.}
The proof of Theorem \ref{thmaux} will be done in several steps. The starting point is to note that by Theorem \ref{area} and Corollary \ref{diameter2}, all the boundary components of a geometrically prime manifold $N$ have diameter at most $\sqrt{\frac{2}{3}}\frac{4\pi}{\sqrt\Lambda_0}$, and area at most $\frac{32\pi}{\sqrt{\Lambda_0}}$. The next step then is to prove that a connected surface $S$ trapped between two equidistant surfaces of comparable radii, $d(x, \Sigma) \in \left[s,s + \rho\right]$ for $x \in S$, has bounded diameter.
This follows by a modification of an argument in Gromov-Lawson \cite[Corollary 10.11]{GrLa83}.

\begin{lemma}\label{lemma:trapped}
Let $M$ be a geometrically prime or a non-orientable geometrically region with large boundary
component $\Sigma$ of diameter $D$.
Let $N_s(\Sigma) = \{x \in M : d(x,\Sigma) \leq s\}$ and suppose $S$ 
is a connected subset of $M$ with
$S \subset N_{s+ \rho}(\Sigma) \setminus N_s(\Sigma)$.
%, where $\rho=\frac{6\pi}{\sqrt{\Lambda_0}}$. 
Then, 
$$\diam(S) \leq \max\{ D+ \sqrt{\frac{2}{3}}\frac{2\pi}{\sqrt{\Lambda_0}}+ 2 \rho, \sqrt{\frac{2}{3}} \frac{12\pi}{\sqrt{\Lambda_0}}+2 \rho \}$$
\end{lemma}

\begin{proof}
Assume that $s \geq \sqrt{\frac{2}{3}}\frac{2\pi}{\sqrt{\Lambda_0}}+ \varepsilon'$.
Note that otherwise we have $\diam(S) \leq  D+ \sqrt{\frac{2}{3}}\frac{2\pi}{\sqrt{\Lambda_0}} +2 \rho + 2 \varepsilon'$.
We will show that the conclusions of the lemma hold under this assumption for every small 
$\varepsilon'>0$ and hence the lemma follows.

Let $x,y$ be two points in $S$. Consider $p_x$ and $p_y$ the points in $\Sigma$ closest to $x$ and $y$, respectively, and let $\gamma_x$ and $\gamma_y$ denote the corresponding minimizing geodesics. Let $\gamma$ denote a curve in $S$
connecting $x$ and $y$ and $\sigma$ denote a curve in $\Sigma$ connecting $p_x$ and $p_y$.

Since $M$ is geometrically prime we have that
every closed surface $\Gamma$ in $M$ is homologous
to a cycle in the boundary of $M$.
Thus, by Poincare-Lefschetz duality
if a closed curve $z$ lies in the interior of $M$,
then it represents a trivial element of $H_1(M, \partial M)$.
Moreover, since all boundary components of $M$
except for $\Sigma$ are spheres we can choose a filling
$Z$ of $z$ with $\partial Z-z \subset \Sigma$.
 
Let $l=\sigma \cup \gamma_x \cup \gamma \cup \gamma_y$
and minimize in the class of surfaces filling $l$
in $(M, \Sigma)$. We obtain a stable minimal surface
$Q$ with $l \subset \partial Q$ and $\partial Q \setminus l \subset \Sigma$.

Choose $\varepsilon \in (0, \varepsilon')$ and 
let $\rho_\varepsilon =\sqrt{\frac{2}{3}} \frac{2 \pi}{\sqrt{\Lambda_0}} + \varepsilon $. Consider surface $\partial N_{s-\rho_\varepsilon }( \Sigma)$.
Without any loss of generality we may assume that the intersection
$\partial N_{s - \rho_\varepsilon}( \Sigma) \cap Q$ is transverse.
Since the distance from $\Sigma$ is a monotonically increasing
function along $\gamma_x$ and $\gamma_y$ we have that
there is a unique connected arc $\alpha$ of 
$\partial N_{s - \rho_\varepsilon}( \Sigma) \cap Q$
connecting $\gamma_x$ to $\gamma_y$. 
Let $p$ denote the point on $\alpha$ that lies at an equal distance
from $\gamma_x$ and $\gamma_y$.

By Theorem \ref{diameter1} we have $d(p, \partial Q) \leq \sqrt{\frac{2}{3}} \frac{2 \pi}{\sqrt{\Lambda_0}}$. Since the distances from $p$ to 
$\gamma$ and $\Sigma$ are larger than that, it follows that
$p$ lies at a distance at most $\sqrt{\frac{2}{3}} \frac{2 \pi}{\sqrt{\Lambda_0}}$ 
from both $\gamma_x$ and $\gamma_y$. Let $x_1$ and $y_1$ denote the closest
points to $p$ on $\gamma_x$ and $\gamma_y$ respectively.
Since $\gamma_x$ is a minimizing geodesic we have
$$d(\Sigma, x_1)+ d(x_1, p) \geq s - \rho_\varepsilon$$
It follows that $d(\Sigma, x_1) \geq  s - \sqrt{\frac{2}{3}}\frac{4 \pi}{\sqrt{\Lambda_0}}+ \varepsilon$. Hence, $d(x_1, x) \leq \frac{4 \pi}{\sqrt{\Lambda_0}}+ \rho + \varepsilon$ and we have the same
inequality for $y_1$ and $y$. We conclude that
$$d(x,y) \leq \sqrt{\frac{2}{3}} \frac{12 \pi}{\sqrt{\Lambda_0}} + 2 \rho + 2 \varepsilon$$
Since we can choose $\varepsilon>0$
to be arbitrarily small this concludes the proof.
\end{proof}

\subsection{Local tree foliation}
We need to modify the mean curvature flow
construction from Theorem \ref{MCF_tree}
so that we have a diameter bound in addition to area and genus bounds. We do this by 
cutting surfaces from the tree foliation into smaller pieces and gluing in stable minimal discs. 

%By ``tripod tree'' in the following lemma we mean a connected graph with three edges and 4 vertices, one of which
%has degree 3.
%In this section, we will state a technical result that will be used in proving that we can do this while keeping control of the genus. 

%Let $U$ be an open set with smooth boundary and $\Sigma \subset \partial U$ a closed (possibly disconnected) surface.  
% \cite[Theorem 4.2]{EcHu91})

\begin{lemma} \label{lem:local foliation}
Let $U\subset M$ be a mean convex region, $\Gamma \subset \partial U$ 
be a connected strictly mean convex surface.
Let $S$ be a stable minimal disc
with $\partial S \subset \Gamma$.
For all sufficiently small $\eps>0$ there exists 
a closed set $U'$, $\Gamma \cup S \subset U' \subset U$
homeomorphic to the closure of the  $\eps$-neighbourhood $\overline{N_{\eps}(\Gamma \cup S)}$,
  tree $T$ and a mean convex tree foliation $f: U' \rightarrow T$ with the following area bounds:
\begin{enumerate}
    \item if $\partial S$ is separating in $\Gamma$,
    then $\Area(f^{-1}(x)) \leq \Area(\Gamma) + \Area(S) $ for all $x \in T$
    and $\Area(f^{-1}(x)) \leq \max\{\Area(\Gamma_1), \Area(\Gamma_2)\}  + \Area(S)$ for 
    a terminal vertex $x$, where $\Gamma_i$, $i=1,2$, denotes connected components
    of $\Gamma \setminus \partial S$;
    \item if $\partial S$ is non-separating in $\Gamma$,
    then $\Area(f^{-1}(x)) \leq \Area(\Gamma) + 2\Area(S) $ for all $x \in T$.
\end{enumerate}

%$\Area(f^{-1}(x)) \leq \Area(\Gamma)$ for all $x \in T$. 
%Moreover, $\partial U' \setminus \Gamma$ is a collection 
%of spheres with outward pointing mean curvature of area  $< \Area(\Gamma)$.
%If $S$ separates $U$ into two disjoint sets $U_1$ and $U_2$, then
%each connected component of $\partial U' \setminus \Gamma$ is contained
%in exactly one of these sets.
\end{lemma}

\begin{proof}
Let $\Gamma$ correspond to the root vertex of tree $T$.
%Let $\{ \gamma_i = \partial S_i \}$ denote the connected components of $\partial S$.
To define a tree foliation we start by flowing surface $\Gamma$ by mean curvature
flow for some very short time $t \in [0,t_1]$. For $t_1$ sufficiently small we have that 
$\gamma=\Gamma_{t_1} \cap S$ is a smooth closed curve.
Let $S' \subset S$ be the minimal disc bounded by $\Gamma_{t_1} \cap S$ and we let
$\Gamma_{t_1} \cup S'$ be the surface corresponding to a degree $3$ vertex of 
the parametrization tree. 
Consider a piecewise smooth immersed surface $\Gamma^1$ 
obtained from $\Gamma_{t_1}$ by cutting it along
$\gamma_1$ and gluing in two copies of $S_1'$. 
If $\partial S$ was separating in $\Gamma$, then $\Gamma_{t_1}$ is the union of two 
piecewise smooth connected immersed surfaces, otherwise it is one connected immersed surface.

Applying mean curvature flow to $\Gamma^1$ will give us a monotone deformation that immediately makes the surface smooth and mean convex (see \cite[Theorem 4.2]{EcHu91}).
If the flow is applied for sufficiently short time we have that that the surfaces are contained
in a small tubular neighbourhood of $\Gamma \cup S$. The area bound follows by the properties
of mean curvature flow.
\end{proof}

\subsection{Existence of tree foliation with controlled area, genus and diameter.}

\begin{theorem}\label{thmaux}
Let $(N^3,h)$ be a geometrically prime or non-orientable geometrically prime
3-manifold of positive scalar curvature $R_h\geq \Lambda_0>0$. Suppose the large connected component $\Sigma$ 
of $\partial N$ is a minimal surface of 
Morse index $1$ and genus $g \leq 2$.
%has genus $g$ and area $\leq \frac{32 \pi}{\Lambda_0}$.
Then, there exists a tree foliation $\{ \Si_x \}_{x \in G}$
such that for any $x\in G$:
\begin{itemize}
	\item[(a)] $\Sigma_x$ has genus at most $g$;
	\item[(b)] $\Area(\Si_x) < \frac{72\pi}{{\Lambda_0}}$ if $N$ is geometrically prime 
	and $\Area(\Si_x) < \frac{84\pi}{{\Lambda_0}}$ if $N$ is non-orientable geometrically prime;
	\item[(c)] $\diam (\Si_x) \leq \sqrt{\frac{2}{3}} \frac{26\pi}{\sqrt{\Lambda_0}}$.
\end{itemize}

%graph
%tree
%$\Gamma = (V, E)$ and a proper map $f:N\rightarrow \Gamma$ such that:
%
%\begin{itemize}
%	\item[(a)]  For every $p\in\Gamma$, $|f^{-1}(p)| \leq \frac{32\pi}{{\Lambda_0}}$ and  ${\rm diam}_N(f^{-1}(p)) \leq \frac{40\pi}{\sqrt{\Lambda_0}}$;
%	\item[(b)] $\Gamma$ is an oriented tree;
%	\item[(c)] The pre-image $f^{-1}(v)$ of each vertex $v\in V$ is either a boundary component of $N$ or a surface with corners in $int(N)$ consisting of the piecewise union of mean-convex and area-minimizing pieces.  
%	%\item[(d)] The union of all pre-images of the edges $e$ in $E$ gives a smooth mean-convex foliation %of $M\setminus f^{-1}(V)$ whose mean curvature vector orientation agrees with the orientation of %$\Gamma$.
%\end{itemize}
\end{theorem}

\begin{proof}

\textbf{1. Orientable case.}
Assume $N$ is geometrically prime.
Suppose $\Sigma$ is a surface of genus $g$, $0 \leq g \leq 2$.
If $g>0$, then we start by cutting $\Sigma$ by minimal discs to reduce its genus to $0$.

By Proposition \ref{free boundary disc} we can find a free boundary stable minimal disc 
$D$ with $\partial D \subset \Sigma$.
Applying Lemma \ref{lem:local foliation} we can define
a tree foliation of a small neighbourhood $U'$ of $\Sigma \cup D$, so that
$\Sigma^1 = \partial U' \setminus \Sigma$ is a surface of genus $g-1$. Performing
this procedure at most two times we obtain a mean convex
sphere $\Sigma^2$.

Observe that by the area bounds
for free boundary stable minimal discs Theorem \ref{area free boundary} and  Lemma \ref{lem:local foliation} we have
$$\Area(\Sigma^2) < \frac{32 \pi}{\Lambda_0}+ 4 \frac{4 \pi}{\Lambda_0} \leq \frac{48 \pi}{\Lambda_0}.$$
To bound the diameter, suppose $x,y \in \Sigma \cup D$. It is easy to see that by Theorem \ref{diameter1} and Corollary \ref{diameter2} $$dist(x,y) \leq \diam(\Sigma) + \sqrt{\frac{2}{3}} \frac{4 \pi}{\sqrt{\Lambda_0}} \leq \sqrt{\frac{2}{3}} \frac{8 \pi}{\sqrt{\Lambda_0}}$$
Applying this twice we obtain that for $\delta>0$ that can be taken to be arbitrarily small for sufficiently small $\eps$
we have

$$\diam(\Sigma^2) \leq \sqrt{\frac{2}{3}} \frac{12 \pi}{\sqrt{\Lambda_0}}+ \delta$$
%After a small inward perturbation we may assume that $\Sigma_2$
%is mean convex (\cite[Theorem 4.2]{EcHu91}). 
(Recall that $\diam$ is the extrinsic diameter.)

Let $\tilde{N}$ denote the subset of $N$ bounded
by mean convex sphere $\Sigma^2$ and the union of stable minimal spheres.

Given a subset $X \subset \tilde{N}$ let 
$$D(X) = \max\{dist(x, \Sigma^2): x \in X\}- \min \{dist(x, \Sigma^2): x \in X\} $$
Pick a minimal covering of $\tilde{N}$ by balls $\mathcal{B}=\{ B_k \}$ of radius $\frac{\pi}{4 \sqrt{\Lambda_0}}$.
Let $N' \subset \tilde{N}$ be a geometrically prime 3-manifold contained in $\tilde{N}$ with
mean convex sphere $\Sigma'$ as the large boundary component and satisfying
\begin{enumerate}
    \item $\Area(\Sigma') \leq \Area(\Sigma^2)$;
    \item $D(\Sigma') < \sqrt{\frac{2}{3}} \frac{6 \pi}{\sqrt{\Lambda_0}}- 2\delta$.
\end{enumerate}
We will prove that every $N' \subset \tilde{N}$ as above
admits a tree foliation, satisfying bounds (a), (b), (c). The proof is by induction on the number of balls in $\mathcal{B}$ necessary to cover $N'$.
 Observe that if $N'$ is contained in 
one ball of the covering, then we can apply Proposition \ref{MCF_tree} to obtain the desired tree 
foliation of $N'$. 

For a fixed $n \in \mathbb{N}$ assume that the tree foliation exists for all $N'$
as above that can be covered by at most $n-1$ balls from $\mathcal{B}$.
Fix $N'$ that has a minimal covering by $n$ balls from $\mathcal{B}$.

Apply Proposition \ref{MCF_tree} to $N'$
to obtain a tree foliation $\{\Sigma_x\}_{x \in G}$,
with the corresponding maps $p:N \rightarrow G$
and $s: G \rightarrow [0,T]$ with fibers of controlled
area.
Observe that $\tilde{D}(x) = D(\Sigma_x)$ is continuous on the interior of every
edge $E$ of $G$. Also, given a vertex $v$ of $G$, an adjacent
edge $E_1$ directed towards $v$ and adjacent edge $E_2$
directed away from $v$ we have 
$$ \lim_{E_2 \ni t \rightarrow v}  \tilde{D}(t) \leq \lim_{E_1 \ni t \rightarrow v} \tilde{D}(t)+
\varepsilon,$$
for an $\varepsilon \in (0, \delta)$ that can be chosen to be small. This follows since each vertex in the construction
of $G$ from the proof of Proposition \ref{MCF_tree}
corresponds to gluing a very small disc in
a cylindrical region. If $r<\varepsilon$
denotes the radius of the disc, then the surface obtained 
by performing a surgery along the disc will 
 have points that are at most distance  $r$
further away from $\Sigma^2$.

If $\tilde{D}(G) \subset [0, \sqrt{\frac{2}{3}} \frac{6\pi}{\sqrt{\Lambda_0}}-2\delta]$, then we can apply Lemma
\ref{lemma:trapped} with $\rho = \sqrt{\frac{2}{3}}\frac{6\pi}{\sqrt{\Lambda_0}}-2\delta$
to obtain 
\begin{equation} \label{eq:diameter bound}
    \diam(\Sigma_x) < \sqrt{\frac{2}{3}} \frac{26\pi}{\sqrt{\Lambda_0}} -\delta
\end{equation}
Thus, without any loss of generality we may assume that $\tilde{D}(x)> \sqrt{\frac{2}{3}} \frac{6\pi}{\sqrt{\Lambda_0}}-2\delta$
for some $x \in G$.
For some $\delta' \in (\varepsilon, 2 \varepsilon)$ to be picked later, let $G_1$ denote the connected component of 
$\tilde{D}^{-1}([0, \sqrt{\frac{2}{3}}\frac{6\pi}{\sqrt{\Lambda_0}}-2 \delta- \delta'])$
that contains the root vertex $q$ of $G$. We choose $\delta'$ so that vertices of $G$ do not lie
in the boundary of $G_1$.
By construction, we have that $N' \setminus p^{-1}(G_1)$ is a disjoint union of
geometrically prime 3-manifolds $V_1, ..., V_k$
with corresponding large boundary components 
$\Gamma_1, ..., \Gamma_k$.
By Lemma \ref{lemma:trapped} the diameters of $\Sigma_x$ for
$x \in G_1$ satisfy bound (\ref{eq:diameter bound}).

Hence, it is enough to prove existence of tree foliations of $V_i$, for $i=1,...,k$, with the desired bound on genus, area and diameter.

Fix $i$. 
%Let $x$ and $y$ denote points 
%of $\Sigma_i$ that attain the
%maximal and minimal distances:
%$$d(x,\Sigma)=\max_{\Sigma_i} d(\cdot,\Sigma),\quad d(y,\Sigma)=\min_{\Sigma_i} %d(\cdot,\Sigma).
By construction we have that $D(\Gamma_i)= \sqrt{\frac{2}{3}}\frac{6\pi}{\sqrt{\Lambda_0}}-2 \delta -\delta'$.
Let $d_i = dist(\Gamma_i, \Sigma^2)$ and define $\gamma = \{x \in \Gamma_i: \,\, d(x,\Sigma^2)=r'\}$, where $r' \in [d_i+ \sqrt{\frac{2}{3}}\frac{3\pi}{\sqrt{\Lambda_0}}- \delta',  d_i+\sqrt{\frac{2}{3}}\frac{3\pi}{\sqrt{\Lambda_0}}]$ is chosen 
so that $\gamma$ is a finite
collection of closed curves. Let $S = \cup_{j=1}^m S_j \subset N_i$ denote the disjoint union of area minimazing discs stable minimal discs
with $\partial S = \gamma$. By Lemma \ref{lem:local foliation} 
 a mean convex tree foliation of a small 
neighbourood $U$ of $\Gamma_i \cup S$ with $\partial U \setminus \Gamma_i = \bigcup \Sigma_l$. 

By Lemma \ref{diameter1}  
for every $x \in S$ we have that
$d(x, \gamma) \leq \sqrt{\frac{2}{3}}\frac{2\pi}{\sqrt{\Lambda_0}}$. 
Hence, by Lemmas \ref{lem:local foliation} and \ref{lemma:trapped} it follows that each connected component $\Sigma_l$ 
 satisfies
\begin{align*}
        \Area(\Sigma_l) & \leq \Area(\Gamma_i) \leq \Area(\Sigma^2)\\
        D(\Sigma_l)< &  \sqrt{\frac{2}{3}} \frac{5 \pi}{\sqrt{\Lambda_0}}  \\
    \diam(\Sigma_l) & \leq \sqrt{\frac{2}{3}} \frac{24\pi}{\sqrt{\Lambda_0}}
\end{align*}
Let $U_l \subset V_i$ denote the geometrically prime 3-manifold
with large boundary component $\Sigma_l$. We claim that the number of balls from
$\mathcal{B}$ necessary to cover $U_l$ must be smaller 
than $n$. Indeed, by construction there exists a point
$x \in \Gamma_i \setminus U_l \subset V_i$
with $dist(x, U_l)> \frac{\pi}{\sqrt{\Lambda_0}}-\delta'>\frac{\pi}{2\sqrt{\Lambda_0}}$. In particular,
we can remove a ball $B'$ containing $x$ from the covering of $U_l$. 
Hence, we can apply the inductive assumption to $U_l$ to obtain
a mean convex tree foliation with the desired diameter bounds. The worst case for the area bound is
for surfaces in the foliaiton that correspond to degree 3 vertices with area bounded by
$\frac{3}{2}\Area(\Sigma^2) \leq \frac{72\pi}{{\Lambda_0}}$.

\textbf{2. Non-orientable case.}
Now suppose $N$ is non-orientable geometrically prime.
Our argument is similar to the orientable case.
We deform the two-sided index 1 projective plane $\Sigma$,
pushing it to the inside using the first eigenfunction of
the Jacobi operator to obtain a mean convex surface $\Sigma^2$.
We pick a minimal covering $\mathcal{B}=\{B_k\}$ of $N$
by balls of radius $\frac{\pi}{4 \sqrt{\Lambda_0}}$ and prove existence of the desired tree foliation
for all non-orientable geometrically prime subset $N' \subset N$ with two-sided mean
convex projective plane $\Sigma'$ as the large boundary component satisfying
\begin{enumerate}
    \item $\Area(\Sigma') \leq \Area(\Sigma^2)$;
    \item $D(\Sigma') < \sqrt{\frac{2}{3}} \frac{6 \pi}{\sqrt{\Lambda_0}}- 2 \delta$.
\end{enumerate}
The proof proceeds by induction on the minimal number of balls in $\mathcal{B}$
necessary to cover $N'$. For a fixed integer $n$ assume that the tree foliation exists for all $N'$ that can be covered by at most $n-1$ balls. Fix $N'$ that has a minimal covering by $n$ balls in $\mathcal{B}$. As in the orientable case we can define a tree foliation of $N'$ with the desired control on the area, but possibly not on the diameter.
%If $D(\Sigma_x)$ becomes larger than $\sqrt{\frac{2}{3}}\frac{6\pi}{\sqrt{\Lambda_0}}$ for some surfaces in the foliation, then as in the orientable case we can find 3-manifolds 
Exactly as in the orientable case we can reduce the problem to constructing a foliation with diameter and area bounds of 3-manifolds $V_1, ..., V_k \subset N'$
with corresponding large boundary components 
$\Gamma_1, ..., \Gamma_k$ with $D(\Gamma_i)= \sqrt{\frac{2}{3}}\frac{6\pi}{\sqrt{\Lambda_0}}-\delta'$ and $\Area(\Gamma_i) < \Area(\Sigma)$. Moreover, it follows from the construction that exactly one of these manifolds (say, $V_1$) is non-orientable geometrically prime with $\Gamma_1 \cong \RP^2$ and all other $V_i$ are geometrically prime with $\Gamma_i \cong S^2$. To define the tree foliation of $V_i$, $i>1$, we proceed as in the orientable case. For 
$V_1$ define a collection of simple closed curves $\gamma \subset \Gamma_1$, $\gamma = \{x \in \Gamma_i: \,\, d(x,\Sigma^2)=r'\}$, where $r' \in [d_1+ \sqrt{\frac{2}{3}}\frac{3\pi}{\sqrt{\Lambda_0}}- \delta',  d_1+\sqrt{\frac{2}{3}}\frac{3\pi}{\sqrt{\Lambda_0}}]$ and $d_1 = dist(\Gamma_1, \Sigma^2)$. Since $\gamma$ is separating in $\Gamma_1$ every connected component $\gamma_j$ of $\gamma$ bounds a disc $D_j$ on one side (and Mobius band on the other side) inside $\Gamma_1$. Let $C_j$ denote a stable minimal disc in $V_1$ obtained by minimizing area in the isotopy class of $D_j$ inside $V_1$ with fixed boundary $\gamma_j$. Discs $C_j$ obtained this way will all be disjoint. We apply Lemma \ref{lem:local foliation} to define a local tree foliation correponding to successively cutting $\Gamma_1$ by discs $C_j$.

In the end we obtain a collection of 2-spheres $\{S_l \}$ bounding geometrically prime regions and a projective plane $\Gamma_1'$ bounding a non-orientable geometrically prime region $V_1'$. By construction we have that $\Area(\Gamma_1') \leq \Area(\Gamma_1) \leq \frac{28\pi}{\Lambda_0}$ and $V_1'$ can be covered by at most $(n-1)$ balls from collection $\mathcal{B}$. By inductive assumption we obtain the desired foliation of $V_1'$. For a sphere $S_l$ bounding a geometrically prime region $W_l$ we observe that $\Area(S_l) < 2 \Area(\Gamma_1) \leq \frac{56\pi}{\Lambda_0}$. We then apply the argument for the orientable case to define the tree foliation for each $W_l$. The worst area bound will be for the case of a degree $3$ vertex corresponding to a sphere of area $\leq \frac{56\pi}{\Lambda_0}$ and a minimizing disc with area at most half of that. Hence, we obtain bounds 
\end{proof}

\begin{proof}[Proof of Theorems \ref{thmmaincompact}]
%and \ref{thmmaincompact2}]
By \cite{White:bumpy} we may assume that
the metric on $M$ is bumpy. By Theorem \ref{decomposition} there exists a collection
of minimal surfaces $\bigcup S_k$, so that
$M \setminus \bigcup S_k$ is a union of  
geometrically prime or non-orientable geometrically prime manifolds $\{N_i\}$.
For each $i$, the large boundary component $\Sigma_i$ of $N_i$ 
is an index $1$ minimal surface of genus 
at most $g$. We have that $g$ is bounded by the maximal Heegaard genus
of prime manifolds in the prime decomposition of $M$, 
so as explained the proof of Theorem \ref{decomposition}
%by Remark \ref{rk:heegaard} 
we have $g \leq 2$.
By Theorem \ref{area},  
$\Area(\Sigma_i)\leq \frac{32 \pi}{\sqrt{\Lambda_0}}$ for orientable components
and $\Area(\Sigma_i)\leq \frac{28 \pi}{\sqrt{\Lambda_0}}$ for non-orientable components.
Hence, there exists a tree foliation $\{\Sigma_x \}_{x \in G_i}$ of $N_i$ with bounds
on area, diameter and genus as in Theorem \ref{thmaux}.
Let $p_i: N_i \rightarrow G_i$ denote the corresponding maps.

Define graph $G = \bigcup G_i / \sim$, where we 
identify vertices $G_i \ni v \sim w \in G_j$
if $p_i^{-1}(v) = p_j^{-1}(w) \subset N$. 
Define map $P: M \rightarrow G$ by setting
$P = p_i$ on the interiors of $N_i$ and extend the map
to $\bigcup S_k$. We have that $P$ and $G$
have the following properties:
\begin{enumerate}
    \item $P^{-1}(x)$ has genus at most $g \leq 2$;
%    , or at most $0$ if $M \cong S^2 \times S^1$.
%$\max\{3g,1\}$ 
%    and at most $3g$ if $M$ doesn't contain embedded   $\RP^2$.
	\item $\Area(P^{-1}(x)) < \frac{56\pi}{{\Lambda_0}}$;
	\item $\diam (P^{-1}(x)) \leq \sqrt{\frac{2}{3}} \frac{26\pi}{\sqrt{\Lambda_0}}$.
    \item For each edge $E \subset G$ the family $\{P^{-1}(t)\}_{t \in \mathring{E}}$
    gives a smooth foliation of $f^{-1}(\mathring{E})$.
    \item $G$ has vertices of degree 1, 2, or 3, moreover
\begin{itemize}
    \item at each vertex of degree $2$ or $3$,
    $P$ satisfies the description from Definition \ref{def: tree};
    \item at vertices of degree 1, $P^{-1}(v)$
    is either a point or minimal stable $\RP^2 \subset M$.
    In the first case surfaces $P^{-1}(t)$, $E \ni t \rightarrow v$, are spheres shrinking to a point;
    in the second case, $P^{-1}(t) \rightarrow P^{-1}(v)$
    %$E \ni t \rightarrow v$, 
    is a two-sheeted smooth convergence.
\end{itemize}    
\end{enumerate}
%This finishes the proof of Theorem \ref{thmmaincompact2}.

To prove Theorem \ref{thmmaincompact}
we choose a smooth function
$s: G \rightarrow \R$ in general position.
It follows from our construction that we may assume that $s$
has no critical points in the interiors of the edges of $G$.
Also, we may assume that local exrema of $s$ are vertices
of degree $1$ or $2$.
%Note that if $M \cong S^3$, then $G$ is a tree and
%we can choose $s$ so that it does not have 
%critical points in the interiors of the edges of $G$.
%More generally, if $G'$ is a subgraph of $G$ that corresponds to
%a 
Choose a collection of small balls $\{ B_l \}$ in $G$,
each $B_l$ centered at a vertex of $G$.
%or a critical point of $s$.
By general position we
may assume that $g(B_{l_1})$ and $g(B_{l_2})$
are disjoint for $l_1 \neq l_2$.

On $M \setminus P ^{-1} (\bigcup B_l) $
we define function $f(x) = s \circ P (x)$.
Clearly, connected components of
$f^{-1}(x)$  for $x \in M \setminus F ^{-1} (\bigcup B_l) $ will satisfy the desired bounds
on the area, diameter and genus.

%For a critical point $p$ in the interior of an edge
%of $G$ and the interval $B_l=(p -\varepsilon, p+ \varepsilon)$ around $p$
%we have that $P^{-1}(B_l)$ is a small 
%neighborhood of surface $P^{-1}(p)$ in $M$
%smoothly foliated by $\{P^{-1}(t)\}_{t \in B_l}$.
%We may assume that $f(x) = f(y)$ for
%$x, y \in F^{-1}(\partial B_l)$.
%By creating a small neck and then opening it 
%we obtain a Morse foliation of $P^{-1}(B_l)$ 
%starting from $P^{-1}(p-\varepsilon) \cup P^{-1}(p+\varepsilon)$.
%The genus of surfaces in the foliation will be bounded by
%$2\genus(P^{-1}(p))+1 \leq 2g+1 \leq 5$ and the area by
%$2 \Area(P^{-1}(p))+ \varepsilon' \leq \frac{102 \pi}{\Lambda_0}$, where
%$\varepsilon'$ can be made arbitrarily small
%by choosing sufficiently small $\varepsilon$.

For a ball $B_l$ around a vertex of degree $1$ we have that $P^{-1}(t)$, $t\in B_l$,
is a parametrization of smooth two-sheeted convergence of spheres to a minimal $\RP^2$.
By creating a small neck and then opening it 
we obtain a Morse foliation of $P^{-1}(B_l)$.

For a ball $B_l$ around a 
vertex $v$ of degree $2$ or $3$
we proceed as follows. Let $s_1$
denote the boundary point of $\partial B_l$
that lies on an edge directed towards $v$
and $s_2$ denote the set of one or two points
that lie on the edge (or edges) directed away 
from $v$.
We define a Morse foliation of the region between
$P^{-1}(s_1)$ and $P^{-1}(s_2)$
using \cite[Lemma 4.2]{CL2019}.
Recall that $P^{-1}(v)$ is a union of two connected smooth 
surfaces $S_1$ and $S_2$, $S^2$ diffeomorphic to a disc, with $\partial S_2 \subset S_1$.
$P^{-1}(s_2)$ is diffeomorphic to a surface 
obtained from $S_1$ by removing the tubular neighborhood
of $\partial S_2$ and gluing in two copies of $S_2$.
We can fix a Morse function $h:S_2 \rightarrow [0,1]$
with $h^{-1}(0) = \partial S_2$ and $h^{-1}(1) = \{ p\} $ a point in $S^2$,
such that $h$ has exactly one critical point of index $2$ (at $p$).
%and $n_e + 2 \genus(S_2)$ critical points of index $1$,
%where $n_e$ is the number of boundary components of $S_2$.
Using Morse function $h$ in the construction 
of foliation from \cite[Lemma 4.2]{CL2019}
we obtain that the genus of surfaces in the foliation
is bounded by $genus(S)+1\leq 3$ and the area is at most $\Area(S_1) + 2 \Area(S^2)$. The largest upper bounds for area will correspond to degree
$3$ vertices that we constructed in the proof Theorem \ref{thmaux} with
$\Area(S_1) + 2 \Area(S^2)< \frac{96 \pi}{\Lambda_0}$ for orientable 3-manifolds and $<\frac{112 \pi}{\Lambda_0}$ for non-orientable 3-manifolds.

This defines the desired Morse function $f: M \rightarrow \R$.
\end{proof}

\bibliography{bib} 
\bibliographystyle{amsalpha}

\end{document}